\documentclass[letterpaper,twoside,reqno]{amsart}

\usepackage[english]{babel}

\usepackage[T1]{fontenc}
\usepackage{pslatex} % font times

\usepackage{amsthm,amsmath,amssymb,amsfonts,amsrefs,dsfont}
\usepackage{upgreek}
\usepackage{epsfig}

\newtheorem{theorem}{Theorem}
\newtheorem{lemma}[theorem]{Lemma}

\newtheorem{definition}{Definition}

\newtheorem{proposition}[theorem]{Proposition}

\usepackage{soul}
\usepackage{color}

\def\R{\mathbb{R}}
\def\PRPC{PrP\textsuperscript{C} }
\def\PRPSC{PrP\textsuperscript{Sc} }
\def\Ab{A$\upbeta$}
\def\bamyloid{$\upbeta$-amyloid }
\def\bamyloids{$\upbeta$-amyloids }

\def\fin{f^{in}}
\def\uin{u^{in}}
\def\pin{p^{in}}
\def\bin{b^{in}}

\newcommand{\ds}{\displaystyle}

\newcommand{\signmh}{ Mohamed Helal \par  Universit\'e Djillali Liabes, Facult\'e des Sciences, D\'epartement de Math\'ematique \par  Sidi Bel Abbes 22000, Alg\'erie \par E-mail: {\tt mhelal$\_$abbes@yahoo.fr} }
\newcommand{\signlpm}{ Laurent Pujo-Menjouet \par   Universit\'e de Lyon, CNRS UMR 5208, Universit\'e Lyon 1, Institut Camille Jordan \par  43 blvd. du 11 novembre 1918, F-69622 Villeurbanne cedex, France\par  E-mail: {\tt  pujo@math.univ-lyon1.fr}}
\newcommand{\signgfw}{ Glenn F. Webb \par  Department of Mathematics, Vanderbilt University \par  1326 Stevenson Center, Nashville, TN 37240-0001, USA \par  E-mail: \tt {glenn.f.webb@vanderbilt.edu}}
\newcommand{\signeh}{ Erwan Hingant \par  CI\textsuperscript{2}MA, Universidad de Concepci\'on, Chile\par  E-mail: {\tt ehingant@ci2ma.udec.cl }}

\begin{document}

\title[Alzheimer's disease: analysis of a mathematical model]{Alzheimer's disease: analysis of a mathematical model incorporating the role of prions}

\author[M. Helal, E. hingant, L. Pujo-Menjouet \& G. F. Webb]{Mohamed Helal \and Erwan Hingant \and Laurent Pujo-Menjouet \and  Glenn F. Webb}

\date{\today}



\thanks{This work was supported by ANR grant MADCOW no. 08-JCJC-0135-CSD5. E. H. was partially supported by FONDECYT (Grant no. 3130318) and would thanks A. Rambaud for helpfull discussion to improve this paper.}

\begin{abstract}
 We introduce a mathematical model of the \emph{in vivo} progression of Alzhei\-mer's disease with focus on the role of prions in memory impairment. Our model consists of differential equations that describe the dynamic formation of \bamyloid plaques based on the concentrations of \Ab~oligomers, \PRPC proteins, and the \Ab-$\times$-\PRPC complex,  which are hypothesized to be responsible for  synaptic toxicity. We prove the well-posedness of the model and provided stability results for its unique equilibrium, when the polymerization rate of \bamyloid is constant and also when it is described by a  power law.
\end{abstract}

\maketitle

%%% Sommaire %%%

 \setcounter{tocdepth}{2} % Profondeur de la table

 \parskip=0pt

 \tableofcontents

 \parskip=\smallskipamount

%%% Intoduction %%%

\section{Introduction} \label{sec:intro}

\subsection{What is the link between Alzheimer disease and prion proteins?} Alzheimer's disease (AD) is acknowledged as one of the most widespread diseases of age-related dementia with $\approx$ 35.6 million people infected worldwide (World Alzheimer Report 2010 \cite{Wimo2010}).  By the 2050's, this same report has  predicted three or four times more people living with AD. AD affects memory, cognizance, behavior, and eventually leads to death. Apart from the social dysfunction of patients, another notable societal consequence of AD is its economic cost ($\approx$ \$422 billion in 2009 \cite{Wimo2010}). The human and social impact of AD has driven extensive research to understand its causes and to develop effective therapies. Among recent findings are the results that imply cellular prion protein (\PRPC) is connected to memory impairment \cites{Cisse2009,Cisse2011,Gimbel2010,Lauren2009,Nath2012}. This connection is the focus of our modeling here, which we hope will contribute to  understanding the relation
of AD to prions.

The pathogenesis of AD is related to a gradual build-up of \bamyloid (\Ab) plaques in the brain \cites{Duyckaerts2009,Hardy2002}. \bamyloid~plaques are formed from the \Ab~peptides obtained from the amyloid protein precursor (APP) protein cleaved at a displaced position. There exist different forms of \bamyloids, from soluble monomers to  insoluble fibrillar aggregates \cites{Chen2011,Lomakin1996a,Lomakin1997,Urbanc1999,Walsh1997}. It has been revealed that the toxicity depends on the size of these structures and recent evidence suggest that oligomers (small aggregates) play a key role in memory impairment rather than \bamyloid~plaques (larger aggregates) formed in the brain \cite{Selkoe2008}. More specifically, \Ab~oligomers  cause memory impairment \emph{via} synaptic toxicity onto neurons.  This phenomenon seems to be induced by a membrane receptor, and there is evidence that this rogue agent is the \PRPC protein \cites{Nygaard2009,Zou2011,Resenberger2011,Gimbel2010,Lauren2009} We note that this protein,
when misfolded in a pathological form called \PRPSC,
is responsible for Creutzfeldt-Jacob disease. Indeed, it is believed that there is a high affinity between \PRPC and \Ab~oligomers, at least theoretically \cite{Gallion2012}. Moreover, the prion protein has also been identified as an APP regulator, which confirms that both are highly related \cites{Nygaard2009,Vincent2010}. This discovery offers a new therapeutic target to recover memory in AD patients, or at least slow memory depletion \cites{Freir2011,CHung2010}.

\subsection{What is  our objective?}
Our objective here is to introduce and study a new \emph{in vivo}  model of AD evolution mediated by \PRPC proteins. To the best of our knowledge, no model such as the one proposed here, has yet been advanced. There exist a variety of models specifically designed for Alzheimer's disease and their treatment, such as in \cites{Achdou2012,Calvez2009,Calvez2010,Craft2005,Craft2002,Gabriel2011,Greer2006,Greer2007,Laurencot2012,Pruss2006,Simonett2006}.
%ADD REFERENCES?
Nevertheless, the prion protein has never been taken into account in the way we formulate here, and our model could helpful in designing new experiments and treatments..

 This paper is organized as follows. We present the model in section \ref{sec:themodel}, and provide a well-posedness result in the particular case that \bamyloids  are formed at a constant rate. In  section \ref{power} we provide a theoretical study of our model in a more general context with a power law rate of polymerization, \emph{i.e.} the polymerization or build-up rate depends on \bamyloid~plaque size.
%Finally, in the fourth section we propose a numerical scheme for the system and test it in particular cases.

\section{The model} \label{sec:themodel}

\subsection{A model for beta-amyloid formation with prions} \label{ssec:model}

The model deals with four different species. First, the concentration of \Ab~oligomers consisting of  aggregates of a few \Ab~peptides; second, the concentration of the \PRPC~protein; third,  the concentration of the complex formed from one \Ab~oligomer binding onto one \PRPC~protein. These quantities are soluble and their concentration will be described in terms of ordinary differential equations. Fourth, we have the insoluble \bamyloid~plaques described by a density according to their size $x$. This approach is standard in modeling prion proliferation phenomena (see for instance \cites{Greer2006,Calvez2009,Prigent2012}).  Note that the size $x$ is an abstract variable that could be the volume of the aggregate. Here, however, we view aggregates as fibrils that lengthen in one dimension. The size variable $x$ thus belongs to the interval $(x_0, +\infty)$, where $x_0>0$ stands for a critical size below which the plaques cannot form. To summarize we denote, for $x \in (x_0, +\infty)$ and $t \geq 0$,
\begin{itemize}
\item $f(t,x)\geq 0$ : the density of \bamyloid~plaques  of size $x $ at time $t$, \vspace{0.6em}
\item $u(t)\geq 0$: the concentration of soluble \Ab~oligomers (unbounded oligomers) at time $t $,\vspace{0.6em}
\item $p(t)\geq 0$: the concentration of soluble cellular prion proteins \PRPC~at time $t$,\vspace{0.6em}
\item $b(t)\geq 0$: the concentration of \Ab-$\times$-\PRPC complex (bounded oligomers) at time $t$.
\end{itemize}

\begin{figure}[!ht]
\centering
\includegraphics[width=129mm]{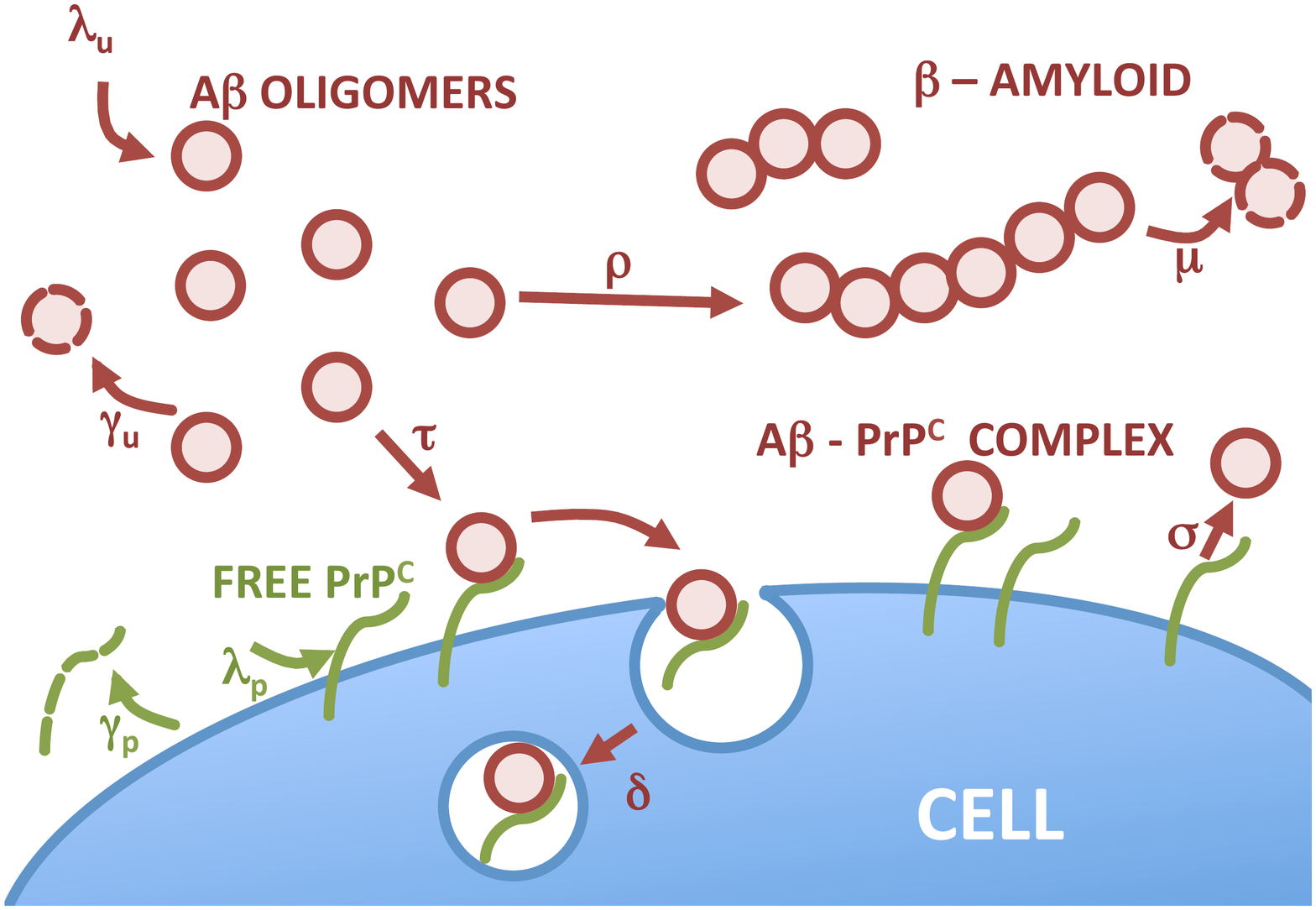}%\includegraphics[width=5.0in,height=3.5in]{prion-alzheimer-2013-figure.pdf}
\caption{Schematic diagram of the evolution processes of \bamyloid plaques, \Ab~oligomers (bounded and unbounded), and \PRPC in the model.}
\label{Figure1}
\end{figure}

 Note that \bamyloid~plaques  are formed from the clustering of \Ab~oligomers. The rate of agglomeration depends on the concentration of soluble oligomers and the structure of the amyloid which is linked to its size. It occurs in a mass action between plaques and oligomers at a nonnegative rate given by $\rho(x)$, where $x$ is the size of the plaque.  This is the reason why the intentionally misused word ``size'' considered here (and described above) accounts for the mass of \Ab~oligomers that form the polymer. We assume indeed, that the mass of one oligomer is given by a ``sufficiently small'' parameter $\varepsilon >0 $. Thus, the number of oligomers in a plaque of mass $x \, > 0$ is $x/\varepsilon$ which justifies our assumption that the size of plaques is a continuum. Moreover, amyloids have a critical size  $x_0 = \varepsilon n \, > 0$,
where $n \, \in \mathbb{N}^{*}$ is the number of oligomers in the critical plaque size. The amyloids are prone to  be damaged at a nonnegative rate $\mu$, possibly dependent on the size $x$ of the plaques.  All the parameters for \Ab~oligomers, \PRPC, and \bamyloid~plaques, such as production, binding  and degradation rates, are nonnegative and described in table \ref{tab:parameter}.

Then, writing evolution equations for these four quantities, we obtain
\begin{align}%[left=\empheqlbrace]{align}
&\frac{\partial}{\partial t} f(x,t) + u(t)\frac{\partial}{\partial x}\big[ \rho(x)f(x,t) \big] = -\mu(x)f(x,t) \quad \text{on }  (x_0,+\infty)\times(0,+\infty), \label{eq:seq1} \vphantom{\int_{x_0}^{+\infty}}\\
&\dot{u} = \lambda_u - \gamma_u u -\tau u p + \sigma b - n N(u) - \frac 1 \varepsilon u \int_{x_0}^{+\infty} \rho(x) f(x,t) d x \quad \text{on }  (0,+\infty), \label{eq:seq2} \vphantom{\int_{x_0}^{+\infty}}  \\
&\dot{p} = \lambda_p - \gamma_p p -\tau u p + \sigma b \quad \text{on }  (0,+\infty), \label{eq:seq3} \vphantom{\int_{x_0}^{+\infty}}\\
&\dot{b} = \tau u p - (\sigma + \delta) b \quad \text{on }  (0,+\infty). \label{eq:seq4} \vphantom{\int_{x_0}^{+\infty}}
\end{align}
The term $N$ accounts for the formation rate of a new \bamyloid plaque with size $x_0$ from the \Ab~oligo\-mers. In order to balance this term, we add the boundary condition
\begin{equation} \label{eq:boundary}
u(t) \rho(x_0) f(x_0,t) = N(u(t)),\quad   t\geq 0.
\end{equation}
The integral in the right-hand side of equation \eqref{eq:seq2} is the total polymerization with parameters$ 1/\varepsilon$, since $\mathrm d x/ \varepsilon$ counts the number of oligomers into a unit of length $\mathrm d x$. Finally, the problem is completed with nonnegative initial data, a function $\fin\geq 0$ and  $\uin,\ \pin,\ \bin \geq 0$, such that at time $t=0$
\begin{equation}
f(\cdot,t=0) = \fin \quad  \text{on } (x_0,+\infty), \label{eq:initial_f}
\end{equation}
and
\begin{equation}
u(t=0)=\uin, \quad p(t=0) = \pin \quad \text{and} \quad b(t=0) = \bin.  \label{eq:initial_upb}
\end{equation}

The above system (\ref{eq:seq1}-\ref{eq:boundary}) involves two formal balance laws: the first one for prion proteins
\begin{equation*}
 \frac{d}{dt} \left(b + p \right) = \lambda_p - \gamma_p p - \delta b,
\end{equation*}
and the second for \Ab~oligomers
\begin{equation*}
 \frac{d}{dt} \left(b + u + \frac 1 \varepsilon \int_{x_0}^{+\infty} x f dx \right) = \lambda_u - \gamma_u u - \delta b - \frac 1 \varepsilon \int_{x_0}^{+\infty} x \mu f dx.
\end{equation*}
The total concentrations of both evolve in time according to the production and degradation rates. In figure \ref{Figure1} we give a schematic representation of these processes.

 \begin{table}[t]
 \footnotesize\centering
\begin{tabular}{lll}
  Parameter/Variable	& Definition 	& Unit  \smallskip\\
\hline\noalign{\smallskip}
	$t$			& Time 							&  days 		\smallskip\\
	$x$ 			&  size of  \bamyloid plaques 				&  -- 		\smallskip\\
	$x_0$ 			& Critical size of \bamyloid plaques 			&  --  		\smallskip\\
	$n$ 			& Number of oligomers in a plaque of size $x_0$		&  --  			\smallskip\\
	$\varepsilon$ 		& Mass of one oligomer 					&  --  	\smallskip\\
\hline\noalign{\smallskip}
	$\lambda_u$ 	& Source of \Ab~oligomers 					&  days$^{-1}$ 						\smallskip\\
	$\gamma_u$  		& Degradation rate of \Ab~oligomers 			&  days$^{-1}$ 						\smallskip\\
\hline\noalign{\smallskip}
	$\lambda_p$ 		& Source of PrP$^C$  							&  days$^{-1}$ 						\smallskip\\
	$\gamma_p$ 			& Degradation rate of PrP$^C$ 					&  days$^{-1}$ 						\smallskip\\
\hline\noalign{\smallskip}
	$\tau$ 				& Binding rate of \Ab~oligomers onto \PRPC 		&  days$^{-1}$ 						\smallskip\\
	$\sigma$ 			& Unbinding rate of   \Ab-$\times$-\PRPC  		&  days$^{-1}$						\smallskip\\
	$\delta$ 			& Degradation rate of \Ab-$\times$-\PRPC 		&  days$^{-1}$						\smallskip\\
\hline\noalign{\smallskip}
	$\rho(x)$ 			& Conversion rate of oligomers into a plaque 	&  (SAF/sq)$^{-1 \;*} \cdot$days$^{-1}$ 	\smallskip\\
	$\mu(x) $ 			& Degradation rate of a plaque 					&  days$^{-1}$ 						\smallskip\\
\hline

\end{tabular}
\label{tab:parameter}
\caption{Parameter description of the model. $^*$ SAF/sq  means Scrapie-Associated Fibrils per square unit and is explained in detail by Rubenstein et al. \cite{RUB1991} (we consider plaques as being fibrils here).}
\end{table}

%%%

\subsection{An associated ODE system}\label{subsec:associate_ode}

%%%

In this section we investigate constant polymerization and degradation rates, \emph{i.e}, rates independent of the size of the plaque involved in the process. This first approach is biologically less realistic, but technically more tractable, yet still quite challenging for an analytical study of the problem. In  section \ref{power}, the polymerization rate $\rho$ will be taken more realistically as a power of $x$. Here we assume that $\rho(x) := \rho \text{ and } \mu(x):= \mu$
are positive constants. Moreover, without loss of generality, we let $\varepsilon = 1$, which only requires a rescaling of the units in  the equations. Then, we assume a pre-equilibrium hypothesis for the formation of \bamyloid plaques, as formulated in \cite{Portet2009} for filaments, by setting
$N(u) = \alpha u^n$.
The formation rate is given by $\alpha>0$ and the number of oligomers necessary to form a new plaque is an integer, $n\geq1$. With these assumptions we are able to close the system (\ref{eq:seq1}-\ref{eq:seq4}) with respect to \eqref{eq:boundary} into a system of four differential equations. Indeed, integrating \eqref{eq:seq1} over $(x_0,+\infty)$ we get formally an equation over the quantity of amyloids at time $t\geq0$
\begin{equation*}
A(t) =\int_{x_0}^{+\infty} f(x,t) dx.
\end{equation*}
This method has already been used on the prion model in \cite{Greer2006}. Now the problem reads, for $t\geq0$,
\begin{align}
&\dot A = \alpha u^n - \mu A,  \label{eq:edo1} \vphantom{\int_0^\infty}\\
&\dot{u} = \lambda_u - \gamma_u u -\tau u p + \sigma b - \alpha n  u^n - \rho u A, \label{eq:edo2}\vphantom{\int_0^\infty}\\
&\dot{p} = \lambda_p - \gamma_p p -\tau u p + \sigma b, \label{eq:edo3}\vphantom{\int_0^\infty}\\
&\dot{b} = \tau u p - (\sigma + \delta) b. \label{eq:edo4}\vphantom{\int_0^\infty}
\end{align}
The mass of \bamyloid plaques is given by $M(t) = \int_{x_0}^{+\infty} x f(x,t) dx$  which satisfies an equation (formal integration of \eqref{eq:seq1}) that can be solved independently, since
\begin{equation} \label{eq:edo_mass}
\dot M =  n \alpha u^n + \rho u A - \mu M.
\end{equation}
Notice that initial conditions for $A$ and $M$ are given by $A^{in}=\int_{x_0}^{+\infty}  \fin(x) dx$ and $M^{in} = \int_{x_0}^{+\infty} x \fin(x) dx$, while the initial conditions for $u$, $p$ and $b$ are unchanged.

The next subsection is devoted to the analysis of the system (\ref{eq:edo1}-\ref{eq:edo4}).

\subsection{Well-posedness and stability of the ODE system}\label{subsec:stability_edo}

We prove in the following proposition the positivity,  existence, and uniqueness of a global solution to the system (\ref{eq:edo1}-\ref{eq:edo4}) with classical techniques from the theory of ordinary differential equations(\cite{Khalil2002}).
\begin{proposition}[Well-posedness] \label{prop:exist}
Assume  $\lambda_u$, $\lambda_p$, $\gamma_u$, $\gamma_p$, $\tau$, $\sigma$, $\delta$, $\rho$ and $\mu$ are positive, and  let $n\geq 1$ be an integer. For any $(A^{in},\uin,\pin,\bin)\in\R^4_+$ there exists a unique nonnegative bounded solution $(A,u,p,b)$ to the system (\ref{eq:edo1}-\ref{eq:edo4}) defined for all time $t>0$, \emph{i.e}, the solution $A$, $u$, $p$ and $b$ belong to $\mathcal{C}^1_b(\R_+)$ and remains in the stable subset
\begin{equation} \label{eq:stableset}
S=\Bigg\{(A,u,p, b)\in\mathbb{R}^4_+:\;  nA+u+p+2b  \leq nA^{in}+\uin+\pin+2\bin+\frac{\lambda}{m}\Bigg\}
\end{equation}
with $\lambda=\lambda_u+\lambda_p$ and $m=\min\{\mu,\gamma_u,\gamma_p,\delta\}$. Furthermore, let $M(t=0)=M^{in}\geq0$, and then there exist a unique nonnegative solution $M$ to \eqref{eq:edo_mass}, defined for all time $t>0$.
\end{proposition}
\begin{proof}
Let $F: \R^4 \mapsto \R^4$ be given by
\begin{equation*}
F(A,u,p,b) = \left( \begin{aligned}
&F_1 := \alpha u^n - \mu A \\
&F_2:= \lambda_u - \gamma_u u -\tau u p + \sigma b - \alpha n  u^n - \rho u A \\
&F_3:=\lambda_p - \gamma_p p -\tau u p + \sigma b \\
&F_4:=\tau u p - (\sigma + \delta) b
\end{aligned} \right).
\end{equation*}
$F$ is obviously $\mathcal{C}^1$ and locally Lipschitz continuous on $\R^4$. Moreover, if $(A,u,p,b)\in\R^4_+$, $F_1 \geq 0$ when $A=0$,  $F_2 \geq 0$ when $u=0$,  $F_3 \geq 0$ when $p=0$, and  $F_4 \geq 0$ when $b=0$. Thus, the system is quasi-positive and  the solution remains in $\R^4_+$. Finally, we remark that
\begin{equation*}
\frac{d}{d t} \left(nA + u + p + 2b \right) \leq \lambda - m \left(nA + u + p + 2b \right),
\end{equation*}
with $\lambda = \lambda_u + \lambda_p$ and $m = \min \left\{\mu,\gamma_u,\gamma_p,\delta\right\}>0$, and Gronwall's lemma ensures that
\begin{equation*}
nA(t) + u(t) + p(t) + 2b(t)  \leq nA^{in} + \uin + \pin + 2\bin + \frac \lambda m.
\end{equation*}
This proves the global existence of a unique nonnegative bounded solution $(A,u,p,$ $b)$. The claim for the mass $M$ is straightforward.
\end{proof}
We next consider the existence of a steady state $A_\infty$, $u_\infty$, $p_\infty$, $b_\infty$ and the asymptotic behavior of solutions to (\ref{eq:edo1}-\ref{eq:edo4}). It is easy to compute the steady state by solving the problem
\begin{align}
& \mu A_\infty - \alpha u_\infty^n = 0 \label{eq:sstate1} \vphantom{\int_0^\infty}\\
& \lambda_u - \gamma_u u_\infty - \tau u_\infty p_\infty + \sigma b_\infty - \alpha n u_\infty^n - \rho u_\infty A_\infty =0 \label{eq:sstate2} \vphantom{\int_0^\infty}\\
& \lambda_p - \gamma_p p_\infty - \tau u_\infty  p_\infty  + \sigma b_\infty= 0 \label{eq:sstate3}\vphantom{\int_0^\infty} \\
&  \tau u_\infty p_\infty - (\delta +\sigma)b_\infty =0  \label{eq:sstate4}\vphantom{\int_0^\infty}
\end{align}
From the structure of the second equation, we cannot give an explicit formula for this problem. To obtain $u_\infty$ we have to solve an algebraic equation, which involves a polynomial of degree $n$. However, we can prove that the solution exists, and then $u_\infty$ is given implicitly.  The next proposition  establishes the  local stability of the steady state..
\begin{theorem}[Linear stability] \label{thm:equilibrium}
Under hypothesis of proposition \ref{prop:exist}, there exists a unique positive steady state $A_\infty$, $u_\infty$, $p_\infty$ and $b_\infty$ to (\ref{eq:edo1}-\ref{eq:edo4}) with
\begin{equation*}
A_\infty = \frac \alpha \mu u_\infty^n, \quad p_\infty = \frac{\lambda_p}{\tau^* u_\infty + \gamma_p}, \quad b_\infty = \frac{1}{\sigma} \frac{\lambda_p (\tau - \tau^*) }{\tau^* u_\infty + \gamma_p} u_\infty,
\end{equation*}
where $\tau^* = \tau (1-\sigma/(\delta+\sigma)$ and  $u_\infty$ is the unique positive root of $Q$, defined by
\begin{equation*}
Q(x) = \gamma_p\lambda_u + a  x  - P(x), \ \mathrm{for \;every} \;  x\geq 0
\end{equation*}
with $a =\tau^* (\lambda_u-\lambda_p) - \gamma_u \gamma_p$ and
\begin{equation*}
P(x) = \tau^*\gamma_u x^2 +  \alpha \gamma_p n x^n  + (\alpha \tau^* n+  \rho \gamma_p  \frac{\alpha}{\mu})x^{n+1} +  \rho \tau^*  \frac{\alpha}{\mu}x^{n+2}
\end{equation*}
Moreover, this equilibrium is locally linearly asymptotically stable.
\end{theorem}
\begin{proof}
First, equation \eqref{eq:sstate1} gives $A_\infty$ with respect to $u_\infty$. Then, combining \eqref{eq:sstate3} and \eqref{eq:sstate4} we get $p_\infty$ and $b_\infty$ as functions of $u_\infty$. Now replacing $p_\infty$ and  $b_\infty$ in \eqref{eq:sstate2} we get  $u_\infty$ as the root of $Q$. It is straightforward that $Q$ has a unique positive root. Indeed, it is the intersection between a line and a monotonic polynomial on the half plane.
Now, we linearize the system in $A_\infty$, $u_\infty$,  $p_\infty$ and $b_\infty$. Let $X = (A,u,p,b)^T$ and the linearized system reads
\begin{equation*}
\frac{d}{d t} X = D X,
\end{equation*}
where
\begin{equation*}
D = \left(\begin{array}{cccc}
	- \mu & \alpha n u_\infty^{n-1} & 0 & 0\\
	- \rho u_\infty & \quad \gamma_u -\tau p_\infty -\alpha n^2 u_\infty^{n-1} -\rho A_\infty \quad  & -\tau u_\infty & \sigma \\
	0 & -\tau p_\infty & -(\gamma_p +\tau u_\infty ) & \sigma\\
	0 &\tau p_\infty & \tau u_\infty & -(\sigma + \delta)
	\end{array}\right).
\end{equation*}
The characteristic polynomial is of the form
\begin{equation*}
P (\lambda) = \lambda^4 + a_1 \lambda^3 + a_2 \lambda^2 + a_3 \lambda + a_4,
\end{equation*}
with the $a_i >0$, $i=1\ldots4$ given in the appendix. Moreover it satisfies
\begin{equation*}
a_1 a_2 a_3 > a_3^2 + a_1^2 a_4.
\end{equation*}
Then, according to the Routh-Hurwitz criterion (see   \cite{Allen2007}*{Th. 4.4, page 150}), all the roots of the characterisic polynomial $P$ are negative or have negative real part, thus the equilibrium is locally asymptotically stable.
\end{proof}

To go further, we give a conditional global stability result when no nucleation is considered, \emph{i.e.}, $\alpha = 0$.

\begin{proposition}[Global stability] \label{prop:global_stab}
Assume that $\alpha=0$. Under the condition
\begin{equation*}
 \left(1+2\frac{\delta+\gamma_u}{\sigma}\right)>\frac{\delta}{2\gamma_p}>\frac{\gamma_p}{\sigma},
\end{equation*}
the unique equilibrium is given by
\begin{equation*}
A_\infty = 0, \quad p_\infty = \frac{\lambda_p}{\tau^* u_\infty + \gamma_p}, \quad b_\infty = \frac{1}{\sigma} \frac{\lambda_p (\tau - \tau^*) }{\tau^* u_\infty + \gamma_p} u_\infty,
\end{equation*}
where $u_\infty$  is the unique positive root of $Q(x) = \gamma_p\lambda_u + a x - \tau^*\lambda_u x^2$, with $a = \tau^* (\lambda_u-\lambda_p) - \gamma_u \gamma_p$. Further, this equilibrium is globally asymptotically stable in the stable subset  $S$ defined in \eqref{eq:stableset}.
\end{proposition}
\begin{proof}
The proof is given by a Lyapunov function $\Phi$ stated in the appendix. It is positive when the condition above is fulfilled and its derivative along the solution to
 the system (\ref{eq:edo1}-\ref{eq:edo4}) is negative definite. Thus, from the LaSalle's invariance principle, we get that under these hypotheses the equilibrium of (\ref{eq:edo1}-\ref{eq:edo4}) is globally asymptotically
 stable.
\end{proof}

\section{A power law polymerization rate}\label{power}

The assumption that the polymerization rate $\rho$ and the degradation rate $\mu$ are constant is not always biologically realistic, as recognized in \cites{Calvez2010,Gabriel2011}. Consequently, we study here the more realistic case  $\rho(x) \sim x^{\theta}$, and in the following we restrict our analysis to $\theta \in(0,1)$. We will see that we are able to obtain a result of existence and uniqueness of solutions for this more general case.

\subsection{Hypotheses and main result}

We are interested in nonnegative solutions to the system (\ref{eq:seq1}-\ref{eq:seq4}) with the boundary condition \eqref{eq:boundary}, completed by initial data \eqref{eq:initial_f} and \eqref{eq:initial_upb}, but with  the new assumption $\rho(x) \sim x^{\theta}$. Moreover, we require that our solution preserves the total mass of \bamyloid in order to be biologically relevant. Hence, the solution $f$ will be sought in the natural space  $L^1(x_0,+\infty;xdx)$, since $xdx$ measures the mass at any time.
Our hypotheses for the system (\ref{eq:seq1}-\ref{eq:seq4}) are
%Thus, in the sequel we will be interested in the following assumptions
%
\begin{align*}
& \mbox{(H1)} \hspace{5pt}  \left| \hspace{5pt}
	\begin{aligned}
	& \fin \in L^1(x_0,+\infty ; xdx), \,
	\fin \geq 0, \ a.e. \ x > x_0. \vphantom{\int_0^\infty}
	\end{aligned}\right. \\
%& \mbox{(H2)} \hspace{5pt} \left|  \hspace{5pt}
%	\begin{aligned}
%	& \rho \mbox{ is a nonnegative function belongs to }  W^{2,\infty}([x_0,\infty))\\
%	& \mbox{and, }\\
%	& \mu \mbox{ is a nonnegative function belongs to } W^{1,\infty}([x_0,\infty)).\\
%	\end{aligned}\right. \\ \\	
& \mbox{(H2)} \hspace{5pt} \left| \hspace{5pt}
	\begin{aligned}
	& \rho \, \geq 0\,, \, \,  \, \rho \, \in \,  W^{2,\infty}([x_0,\infty)), \,
	& \mu  \, \geq 0\,, \, \, \, \mu \, \in \, W^{1,\infty}([x_0,\infty)). \vphantom{\int_0^\infty}\\
	\end{aligned}\right. \\
%&\mbox{(H3)} \hspace{5pt} \left|  \hspace{5pt}
%	\begin{aligned}
%	& N \mbox{ is a nonnegative belongs to } W^{1,\infty}_{loc}(\mathbb R_+)\\
%	& \mbox{such that, } N(0) = 0.
%	\end{aligned}\right.\\ \\
&\mbox{(H3)} \hspace{5pt} \left|  \hspace{5pt}
	\begin{aligned}
	& N \, \geq 0\,, \, \, \, N \, \in \,  W^{1,\infty}_{loc}(\mathbb R_+),  \,
	N(0) = 0. \vphantom{\int_0^\infty}
	\end{aligned}\right.\\
%& \mbox{(H4)} \hspace{5pt} \left|  \hspace{5pt} \mbox{All the parameters in table \ref{tab:parameter} from } x_0 \mbox{ to } \delta \mbox{ are positive.}\right.
& \mbox{(H4)} \hspace{5pt} \left|  \hspace{5pt}  \lambda_u, \, \gamma_u, \, \lambda_p, \, \gamma_p, \, \tau, \, \sigma, \, \delta  \, > \, 0. \vphantom{\int_0^\infty} \right.
\end{align*}
%
%\tblue{
We note that (H2) implies the existence of a constant $C>0$ such that $\rho(x) \leq C x$, with for example, $C =  2 \|\rho'\|_{L^\infty} +  \rho(x_0)/x_0$. For any $x\geq x_0$, we have
$$ \rho(x) \leq \|\rho'\|_{L^\infty} (x +  x_0) + \rho(x_0) \leq   \left( 2 \|\rho'\|_{L^\infty}  + \frac{\rho(x_0)}{x_0} \right)  x.
$$
\noindent  We remark that this kind of regularity of the rate $\rho$ covers the case that  $\rho(x) \sim x^{\theta}$  with $\theta \in (0; 1)$.
Also, (H3) implies the existence of a constant $K_M>0$ such that $N(w)\leq K_M w, \mbox{ for any } w\in[0,M]$.
Further, The nonnegativity of the parameters of table \ref{tab:parameter} (hypothesis (H4)) is a natural assumption with regard to their biological meaning.

%}
We introduce the definition of  a solution to system (\ref{eq:seq1}-\ref{eq:seq4}).
\begin{definition} \label{def:exist}
Consider a function $\fin$ satisfying (H1) and let $\uin$, $\pin$, $\bin$ be three nonnegative real data. Assume that $\rho$, $\mu$, $N$ and all the parameters of table \ref{tab:parameter} verify assumptions (H2) - (H4), and let $T>0$. Then a quadruplet $(f,u,p,b)$ of  nonnegative functions is said to be a \emph{solution} on the interval $(0, T)$ to the system (\ref{eq:seq1}-\ref{eq:seq4}) with the boundary condition \eqref{eq:boundary} and the initial data \eqref{eq:initial_f} and \eqref{eq:initial_upb}, if it satisfies, for any $\varphi \in \mathcal{C}^\infty_c\left([0,T]\times[x_0,+\infty)\right)$ and $t\in(0,T)$
\begin{multline*}
\int_{x_0}^{+\infty} f(x,t)\varphi(x,t) dx  = \int_{x_0}^{+\infty} \fin(x)\varphi(x,0) dx + \int_0^t N(u(s))\varphi(x_0,s) ds \hfill \\
\hfill  + \int_0^t\int_{x_0}^{+\infty} f(x,s)\left[ \frac{\partial}{\partial t}\varphi(x,s) + u(s) \rho(x) \frac{\partial}{\partial x}\varphi(x,s) - \mu(x) \varphi(x,s)\right] dx ds ,
\end{multline*}
and
\begin{align*}
& u(t) = \uin + \int_0^t \left[\lambda_u  - \gamma_u u -\tau u p + \sigma b - x_0 N(u) -  u \int_{x_0}^{+\infty} \rho(x) f(x,s) dx \right] ds, \\
& p(t) = \pin + \int_0^t \left[\lambda_p - \gamma_p p -\tau u p + \sigma b \right] ds, \\
& b(t) = \bin + \int_0^t \left[ \tau u p - (\sigma + \delta) b \right] ds,
\end{align*}
with the regularity
$f\in L^\infty\left(0,T;L^1\left(x_0,+\infty; x dx\right)\right)$ and  $u,p,b \in C^0(0,T)$.
\end{definition}
\begin{theorem}[Well-posedness] \label{thm:wellposed}
Let $\fin$ be a nonnegative function satisfying (H1), let $\uin$, $\pin$ and $\bin$ be nonnegative real numbers, and assume hypothesis (H2) to (H4). Let $T>0$. There exists a unique nonnegative solution $(f,u,p,b)$ to (\ref{eq:seq1}-\ref{eq:seq4}) with \eqref{eq:boundary}  and initial conditions given by \eqref{eq:initial_f} and \eqref{eq:initial_upb}, in the sense of definition \ref{def:exist}, such that
$f\in C^0\left([0,T],L^1(x_0,+\infty;x^r dx)\right)$ for every $r\in [0,1]$,
and $u,p,b \in C^1_b(0,T)$.
\end{theorem}
The proof of the theorem \ref{thm:wellposed} is decomposed into two parts. First,  we study the initial boundary value problem
\begin{align}
 &\frac{\partial}{\partial t} f(x,t) + u(t)\frac{\partial}{\partial x}\big[ \rho(x)f(x,t) \big] = -\mu(x)f(x,t) \quad  \text{on } (x_0,+\infty)\times(0,+\infty), \label{eq:transport} \\
 & \vphantom{\frac{\partial}{\partial t}} u(t)\rho(x_0)f(x_0,t) = N(u(t)), \quad \text{on } (0,+\infty), \label{eq:bdry}\\
 &\vphantom{\frac{\partial}{\partial t}} f(\cdot,t=0) = \fin,\quad \text{on }(x_0,+\infty). \label{eq:initial_f_auto}
\end{align}
We prove in the subsection \ref{subsec:autonomous} the following proposition:
\begin{proposition} \label{lem:auto}
Let $u\in\mathcal C^0_b(\mathbb R_+)$, let $\fin$ satisfy (H1), and assume hypothesis (H2) to (H3). For any $T>0$, there exists a unique nonnegative solution $f$ to (\ref{eq:transport}-\ref{eq:initial_f_auto}) in the sense of distributions, such that
$f\in C^0\left([0,T],L^1(x_0,+\infty;x^r dx)\right)$ for every $r\in [0,1]$.

\end{proposition}
\noindent
The proof is in the spirit of the proof in \cite{Collet2000} for the Lifshitz-Slyozov equation. It consists of a proof based on the concept of a mild solution in the sense of distributions, with the additional requirement of continuity from time into $L^1(xdx)$ space.

The second step of the proof of theorem \ref{thm:wellposed} is performed in subsection \ref{subsec:proof}. Precisely, once we have the existence of a unique density $f$, when $u$ is  given, we are able to construct the operator
\begin{equation} \label{eq:operator}
\begin{array}{rcccl}
\ds S & : & C^0([0,T])^3 & \ds \mapsto & \ds C^0([0,T])^3  \vspace{3mm} \\
& & \ds (u,p,b) & \ds \mapsto &  \ds ( S_u , S_p, S_b) = S(u,p,b),
\end{array}
\end{equation}

\begin{equation*}
\begin{array}{rcl}
 \ds S_u & = & \ds \uin + \int_0^t \left[\lambda_u  - \gamma_u u -\tau u p + \sigma b - x_0 N(u) - u \int_{x_0}^{+\infty} \rho(x) f(x,s) dx \right] ds,\vspace{2mm}\\
 \ds S_p & = & \ds \pin + \int_0^t \left[\lambda_p - \gamma_p p -\tau u p + \sigma b \right] ds, \vspace{2mm} \\
 \ds S_b & = & \ds \bin + \int_0^t \left[ \tau u p - (\sigma + \delta) b \right] ds,
\end{array}
\end{equation*}
where $f$ is the unique solution associated to $u$ given by proposition \ref{lem:auto}. Then, theorem \ref{thm:wellposed} is finally proven in subsection \ref{subsec:proof} applying the Banach fixed point theorem to the operator $S$.

\subsection[The autonomous problem]{Existence of a solution to the autonomous problem}\label{subsec:autonomous}

In the following we let $u\in\mathcal C^0_b(\mathbb R_+)$ and we use the notations
$a(x,t) = u(t) \rho(x)$ and $c(x,t) = - u(t) \rho'(x)$ for \;every $(x,t)\in [x_0,+\infty)\times\mathbb R_+$.
From (H2) and noting that $\rho(x)\leq C x$, we have for any $t>0$
\begin{align}
& a(t,x) \leq A x,  \text{ for } x>x_0, \label{eq:a_slin}\\
& | a(t,x) - a(t,y)| \leq A |x - y|,  \text{ for } x,y>x_0,\label{eq:a_lip}\\
& |c(t,x) | \leq B, \label{eq:B}
\end{align}
where $A = \max\left( C \|u\|_{L^\infty},\|u\|_{L^\infty} \|\rho'\|_{L^\infty}\right)$  and $B =  \|u\|_{L^\infty} \|\rho'\|_{L^\infty(x_0,+\infty)}$. In order to establish the mild formulation of the problem, we define the characteristic which reaches $x\geq x_0$ at time $t\geq 0$, that is, the  solution to
\begin{equation} \label{eq:charact}
\begin{aligned}
 & \frac{d}{ds}X(s;x,t) = a(t,X(s;x,t)),\\
 & \vphantom{\frac{d}{ds}} X(t;x,t) = x.
\end{aligned}
\end{equation}
From property \eqref{eq:a_lip}, their exists a unique characteristic that reaches $(x,t)$.We note that it makes sense  as long as $X(s;x,t)\geq x_0$. Thus, we define the starting time of the characteristic as
\begin{equation*} \label{eq:exit}
 s_0(x,t) := \inf \left\{ s \in [0,t] : X(s;x,t) \geq x_0 \right\}.
\end{equation*}
The characteristic will be defined for any time $s\geq s_0$ and takes its origin from the initial or the boundary condition, respectively, if $s_0 = 0 $ or $s_0>0$. We recall the classical properties of these characteristics
\begin{align*}
& X(s;X(\sigma;x,t),\sigma) = X(s;x,t) \\
& J(s;x,t) := \frac{\partial}{\partial x} X(s;x,t) = \exp\left(\int_s^t c(\sigma,X(\sigma;x,t)) d\sigma \right)\\
& \frac{\partial}{\partial t} X(s;x,t) = - a(t,x) J(s;x,t).
\end{align*}
Also, remarking that $s_0(X(t;x_0,0),t) = 0$, then by monotonicity and continuity of $X$ for any $t>0$, we get
$x\in (x_0,X(t;x_0,0)) \ \iff \ s_0(x,t) \in (0,t)$,
and for any $ x\in(x_0,X(t;x_0,0))$ we have $X(s_0(x,t);x,t)=x_0$. It follows that for every $x \in(x_0,X(t;x_0,0))$
\begin{equation*}
I(x,t)  := - \frac{\partial}{\partial x} s_0 (x,t) = J(s_0(x,t);x,t)/a(s_0(x,t),x_0).
\end{equation*}\
\noindent
Considering the derivative of $f(s,X(s;x,t))$ in $s$, and integrating  over $(s_0,t)$ we obtain the mild formulation of the problem. The mild solution is defined  for $a.e. \ (x,t)\in (x_0,+\infty)\times \mathbb R_+$ by
\begin{equation} \label{eq:mild}
f(x,t) = \left\{
	\begin{array}{l}
		\ds \fin(X(0;x,t)) J(0;x,t) \exp\left(- \int_{0}^t\mu(X(\sigma;x,t))d\sigma\right) \quad  x \geq X(t;x_0,0), \\
 		\ds N(u(s_0(x,t))) I(x,t) \exp\left(- \int_{s_0(x,t)}^t\mu(X(\sigma;x,t))d\sigma\right)\quad  x\in(x_0,X(t;x_0,0)).
	\end{array}\right.
\end{equation}
We infer from the formulation \eqref{eq:mild} that for \emph{a.e} $(x,t)\in[x_0,+\infty)\times \mathbb R_+$, $f$ is nonnegative, since $J$ and $I$ are nonnegative, and $\fin$ satisfies (H1).
We recall some useful properties that are derived in \cite{Collet2000}*{Lemma 1}.
\begin{lemma} \label{lem:carac}
Let $u \in \mathcal C^0_b(\mathbb R_+)$ be a given data and assume that (H2) holds. Then for any $x\geq x_0$ and $t>0$, as long as  the  characteristic curve $s\mapsto X(s;x,t)$ defined in \eqref{eq:charact} exists,  \emph{i.e.}, $s\geq s_0(x,t)$, we have
\begin{align*}
& \text{for }  s_1 \leq s_2 , \ X(s_1;x,t) \leq X(s_2;x,t) \leq X(s_1;x,t) e^{A(s_2-s_1)}\\
& \text{if } x_n \rightarrow +\infty, \text{ then for all } t \geq s \geq 0,\ X(s;x,t) \rightarrow +\infty\\
& \text{for } s \geq t, \ X(s;x,t) \leq x e^{A(s-t)}.
\end{align*}
\end{lemma}
\begin{proof}
We refer to \cite{Collet2000}*{Lemma 1}, where the result follows from the fact that for any $x\geq x_0$, $t>0$ and $s_0(x,t)\leq s_1 \leq s_2 $, we have
\begin{equation*}
x_0 \leq X(s_2;x,t) = X(s_1;x,t) +  \int_{s_1}^{s_2} a(s,X(s;x,t)) d s \leq X(s_1;x,t) +  A \int_{s_1}^{s_2} X(s;x,t) d s,
\end{equation*}
where $A$ is given by \eqref{eq:a_slin}.
\end{proof}
In the sequel we will repeatedly refer to the changes of variables
$$y = X(0;x,t) \text{ over } x \in (X(t,x_0,0),+\infty), \text{ with Jacobian } J(0;x,t),$$
$$s = s_0(x,t) \text{ over } x \in (x_0,X(t;x_0,0)), \text{ with Jacobian } -I(x,t).$$
The first is a $\mathcal C^1$ - diffeomorphism from $(X(t,x_0,0),+\infty)$ into $(x_0,+\infty)$, and the second from $(x_0,X(t;x_0,0))$ into $(0,t)$.
Integrating $f$ defined by \eqref{eq:mild} over $(0,R)$ with $R>X(t;x_0,0)$, using the change of variables above, using lemma \ref{lem:carac}, and taking the limit $R\rightarrow +\infty$, we get
\begin{equation} \label{eq:intx}
\begin{aligned}
\int_{x_0}^{+\infty} x|f(t,x)| dx & \leq \int_{x_0}^{+\infty} X(t;y,0) |\fin(y)| dy + \int_0^t X(t;s,x_0) | N (u(s))| ds \\
& \leq e^{At} \left( \int_{x_0}^{+\infty} y |\fin(y)| dy + \int_0^t x_0 | N (u(s))| ds \right),
\end{aligned}
\end{equation}
where we have split the integral into two parts and uses both the previous changes of variables.  Thus,for any $T>0$, $f\in L^\infty\left(0,T; L^1(x_0,+\infty;x dx)\right)$, and therefore in $L^\infty\left(0,T; L^1(x_0,+\infty; x^r dx)\right)$, for any $r\in[0,1]$. In the next lemma we claim that $f$ defined by \eqref {eq:mild} is a weak solution.
\begin{lemma} \label{prop:weak}
Let $f$ be the mild solution defined by \eqref {eq:mild}. Then for any $t>0$
\begin{multline*}
\int_{x_0}^{+\infty} f(x,t)\varphi(x,t) dx  = \int_{x_0}^{+\infty} \fin(x)\varphi(x,0) dx + \int_0^t N(u(s))\varphi(x_0,s) ds \hfill \\
\hfill  + \int_0^t\int_{x_0}^{+\infty} f(x,s)\left[ \frac{\partial}{\partial t}\varphi(x,s) u(s) \rho(x) \frac{\partial}{\partial x}\varphi(x,s) - \mu(x) \varphi(x,s)\right] dx ds,
\end{multline*}
for all $\varphi \in \mathcal{C}^\infty_c([0,T]\times[x_0,+\infty))$.
\end{lemma}
\begin{proof}
Since $f$ belongs to  $f\in L^\infty\left(0,T; L^1(x_0,+\infty;x dx)\right)$, it is possible to multiply the mild solution $f$ against a test function $\varphi\in \mathcal{C}^\infty_c([0,T]\times[x_0,+\infty))$ and integrate over $(x_0,+\infty)$ to obtain
\begin{multline} \label{eq:I0}
\int_{x_0}^{+\infty} f(x,t) \varphi(x,t) dx = \int_{x_0}^{+\infty} \fin(y) \varphi(X(t;y,0))e^{- \int_{0}^t\mu(X(\sigma;y,0))d\sigma}dy\hfill\\
\hfill - \int_0^t N (u(s)) \varphi(X(t;x_0,s),t) e^{ - \int_{s}^t\mu(X(\sigma;x_0,s))d\sigma} ds,
\end{multline}
by the same change of variable made above for \eqref{eq:intx}. Furthermore, we have
\begin{align} \label{eq:I1}
\int_0^t & \int_{x_0}^{X(s;x_0,0)} f(x,s)  \left[\partial_t \varphi(x,s) + a(s,x)\partial_x \varphi(x,s) -\mu(x) \varphi(x,s)\right] dx ds \nonumber \\
& = \int_0^t \int_{x^0}^{+\infty} \fin(x) \frac{d}{ds} \left( \varphi(X(s;x,0),s)e^{- \int_{0}^s\mu(X(\sigma;x,0))d\sigma} \right) dy ds \nonumber\\
& = \int_{x_0}^{+\infty} \fin(x) \varphi(X(t;x,0),t)e^{- \int_{0}^t \mu(X(\sigma;y,0))d\sigma}  dx   - \int_{x_0}^{+\infty} \fin(x) \varphi(x,0) dx,
\end{align}
still using the change of variable mentioned above and
\begin{align} \label{eq:I2}
\int_0^t & \int_{X(s;x_0,0)}^{\infty} f(x,s)  \left[\partial_t \varphi(x,s) + a(s,x)\partial_x \varphi(x,s) -\mu(x) \varphi(x,s)\right] dx  ds \nonumber \\
& =  -\int_0^t \int_{0}^{s} N(u(z)) \frac{d}{d s} \left( \varphi(X(s;x_0,z),s)e^{- \int_{z}^s\mu(X(\sigma;x_0,z))d\sigma} \right) dz ds \nonumber\\
& =  - \int_0^t  N(u(s))  \varphi(X(t;x_0,s),t)e^{- \int_{s}^t\mu(X(\sigma;x_0,s))d\sigma} dz ds  -  \int_0^t N(u(s)) \varphi(x_0,s) ds.
\end{align}
Finally, combining \eqref{eq:I0}, \eqref{eq:I1} and \eqref{eq:I2} we obtain that $f$ is a weak solution.
\end{proof}
The aim of the following lemma is to prove that the moments of $f$ less than 1 are continuous in time.
\begin{lemma} \label{prop:continuity}
Let hypothesis (H1) to (H3) hold. Let $f$ be the mild solution given by \eqref{eq:mild}. Then for any $T>0$,
\begin{equation*}
f\in C^0\left([0,T],L^1(x_0,+\infty;x^r dx)\right), \quad \mathrm{\;for \;every\;} r\in [0,1].
\end{equation*}
\end{lemma}
\begin{proof}
Let $T>0$ and $r\in[0,1]$, since $f\in L^\infty_{loc}\left(\mathbb R_+, L^1(x_0,+\infty; x^r dx)\right)$, we have for any $t>0$ and $\delta t > 0$ such that $t+\delta t \leq T$
\begin{equation*}
\int_{x_0}^{+\infty} x^r \left|f(x,t+\delta t) - f(x,t)\right| dx   =  I_1 + I_2 + I_3,
\end{equation*}
where
\begin{align*}
& I_1 = \int_{x_0}^{X(t;x_0,0)} x^r \left|f(x,t+\delta t) - f(x,t)\right| dx, \\
& I_2 = \int_{X(t;x_0,0)}^{X(t+\delta t;x_0,0)} x^r \left|f(x,t+\delta t) - f(x,t)\right| dx, \\
& I_3 = \int_{X(t+\delta t;x_0,0)}^{+\infty} x^r \left|f(x,t+\delta t) - f(x,t)\right| dx.
\end{align*}
Our goal is to prove that each term goes to zero when $\delta t$ goes to zero.  We first bound $I_3$, which results from the initial condition, since for $x \geq X(t+\delta t;x_0,0) \geq X(t;x_0,0)$, it  follows that
\begin{multline*}
I_3 = \int_{X(t+\delta t;x_0,0)}^{+\infty} x^r \left|\fin(X(0;x,t+\delta t))J(0;x,t+\delta t)e^{-\int_{0}^{t+\delta t} \mu(X(\sigma;x,t+\delta t))d\sigma} \right. \hfill \\
\hfill \left. - \fin(X(0;x,t))| J(0;x,t)e^{-\int_{0}^t \mu(X(\sigma;x,t))d\sigma}\right| dx.
\end{multline*}
Let $\fin_\varepsilon \, \in \, \mathcal C^\infty_0$ with compact support $supp (\fin_\varepsilon) \subset(0,R_\varepsilon)$ and converge in $L^1([x_0,+\infty),xdx)$ to $\fin$. We write $I_3$ as follows
\begin{equation}
 I_3 = I_3^1 + I_3^2 + I_3^3,
\end{equation}
where
\begin{equation*}
\begin{array}{rcl}
 \ds I_3^1 & = & \ds \int_{X(t+\delta t;x_0,0)}^{+\infty} x^r \big| \fin(X(0;x,t+\delta t))  - \fin_\varepsilon(X(0;x,t+\delta t))\big| \\
 & & \ds \phantom{\int_{X(t+\delta t;x_0,0)}^{+\infty} x^r} \, \times \,  J(0;x,t+\delta t) e^{-\int_{0}^{t+\delta t} \mu(X(\sigma;x,t+\delta t))d\sigma}  dx,  \\
 \ds I_3^2 & = & \ds \int_{X(t+\delta t;x_0,0)}^{+\infty} x^r  \big| \fin_\varepsilon(X(0;x,t+\delta t)) J(0;x,t+\delta t)
 \\
 & &
 \hspace{.5in} \times \, \,
 e^{-\int_{0}^{t+\delta t} \mu(X(\sigma;x,t+\delta t))d\sigma}  \\
& & \ds \phantom{ \int_{X(t+\delta t;x_0,0)}^{+\infty} x^r}  \, - \fin_\varepsilon(X(0;x,t))J(0;x,t)e^{-\int_{0}^t \mu(X(\sigma;x,t))d\sigma} \big| dx , \\
\ds  I_3^3 & = & \ds \int_{X(t+\delta t;x_0,0)}^{+\infty} x^r   | \fin_\varepsilon (X(0;x,t))  - \fin(X(0;x,t))|\\
& & \ds \phantom{\int_{X(t+\delta t;x_0,0)}^{+\infty} x^r  } \, \times \, J(0;x,t) e^{-\int_{0}^t \mu(X(\sigma;x,t))d\sigma} dx.
\end{array}
\end{equation*}
%
% \begin{align*}
% I_3 & \leq \int_{X(t+\delta t;x_0,0)}^{+\infty} x^r | f_0(X(0;x,t+\delta t))  - f_0^\varepsilon(X(0;x,t+\delta t))| J(0;x,t+\delta t) dx \\
% & \quad \begin{aligned}
% + \int_{X(t+\delta t;x_0,0)}^{+\infty} x^r & \left| f_0^\varepsilon(X(0;x,t+\delta t)) J(0;x,t+\delta t)e^{-\int_{0}^t \mu(X(\sigma;x,t+\delta t))d\sigma} \right. \\
% & \left. - f_0^\varepsilon(X(0;x,t))J(0;x,t)e^{-\int_{0}^t \mu(X(\sigma;x,t))d\sigma} \right| dx
% \end{aligned}\\
% & \quad + \int_{X(t;x_0,0)}^{+\infty} x^r   | f_0^\varepsilon (X(0;x,t))  - f_0(X(0;x,t))| J(0;x,t) dx
% \end{align*}
%
%We denote each integral from $I_3^1$ to $I_3^3$ respectively.
Dropping the exponential term, which is bounded by one, and changing of variables $y=X(0;x,t+\delta t)$ in $I^1_3$  and $y = X(0;x,t)$ in $I_3^3$,  we get
\begin{equation} \label{eq:estimI3_1}
I^1_3 + I^3_3 \leq 2e^{AT} \int_{x_0}^{+\infty} y^r | \fin(y)  - \fin_\varepsilon(y)| d y = C_3^1(T,\varepsilon),
\end{equation}
with the help of lemma \ref{lem:carac}. Next we bound $I_3^2$ by
\begin{equation*}
\begin{array}{rcl}
\ds I_3^2  & \leq & \ds \int_{X(t+\delta t;x_0,0)}^{+\infty} x^r  | \fin_\varepsilon(X(0;x,t+\delta t))  - \fin_\varepsilon(X(0;x,t)) | J(0;x,t+\delta t) dx \vspace{2mm}\\
& & \ds + \int_{X(t+\delta t;x_0,0)}^{+\infty} x^r   \fin_\varepsilon(X(0;x,t)) | J(0;x,t+\delta t)  -  J(0;x,t) |  dx \vspace{2mm}\\
& & \ds + \int_{X(t+\delta t;x_0,0)}^{+\infty} x^r   \fin_\varepsilon(X(0;x,t))J(0;x,t) \\
& & \ds \phantom{ \int_{X(t+\delta t;x_0,0)}^{+\infty} x^r} \, \times \, |e^{-\int_{0}^{t+\delta t} \mu(X(\sigma;x,t+\delta t))d\sigma}  -  e^{-\int_{0}^t \mu(X(\sigma;x,t))d\sigma}|  dx,
\end{array}
\end{equation*}
and we denote the integrals by $J_3^1$ to $J_3^3$, respectively.
We remark that $J(0,x,t) \leq e^{BT}$ by \eqref{eq:B} and so
\begin{align}
 J_3^1 & \leq e^{BT}  \|\fin_\varepsilon \|_{L^\infty} \int_{X(t+\delta t;x_0,0)}^{C_\varepsilon} x^r  | X(0;x,t+\delta t) - X(0;x,t)| dx \nonumber\\
 & \leq \delta t e^{BT}  \|\fin_\varepsilon\|_{L^\infty} \int_{X(t+\delta t;x_0,0)}^{C_\varepsilon} x^r \sup_{s\in[t,t+\delta t]} \left|\frac{\partial}{\partial t} X(0;x,s) \right| dx  \nonumber \\
 & \leq \delta t Ae^{2BT}  \|\fin_\varepsilon \|_{L^\infty} \int_{x_0}^{C_\varepsilon} x^{r+1} dx, \label{eq:estimI3_2}
\end{align}
where $C_\varepsilon$ depends on $T$, $A$ and $R_\varepsilon$ \emph{i.e.}, the compact support of $\fin_\varepsilon$. Then
\begin{equation*}
 J_3^2  \leq  e^{BT}  \|\fin_\varepsilon \|_{L^\infty} \int_{X(t+\delta t;x_0,0)}^{R_\varepsilon} x^r | e^{G(t,\delta t, x)} -1 | dx
\end{equation*}
with
\begin{align*}
|G(t,\delta t,x)| & = \Big|\int_0^{t+\delta t} c(\sigma,X(\sigma;x,t+\delta t)) d\sigma - \int_0^{t} c(\sigma,X(\sigma;x,t)) d\sigma \Big|\\
& \leq \int_0^{t+\delta t} \Big| \rho'(X(\sigma;x,t+\delta t))  -  \rho'(X(\sigma;x,t)) \Big| u(\sigma) d\sigma \\
& \phantom{\leq \int_0^{t+\delta t} } + \int_t^{t+\delta t} \Big|c(\sigma,X(\sigma;x,t))\Big| d \sigma .
\end{align*}
Thus, with \eqref{eq:a_slin} and \eqref{eq:B},
\begin{align*}
 |G(t,\delta t,x)| & \leq K \|u\|_{L^\infty} \int_0^{T}  \Big| X(\sigma;x,t+\delta t) - X(\sigma;x,t)\Big| d \sigma + \delta t B\\
& \leq \delta t K \|u\|_{L^\infty} \int_0^{T}  \sup_{s\in[t,t+\delta t]} \left|\frac{\partial}{\partial t} X(\sigma;x,s) \right| d \sigma + \delta t B\\
& \leq \delta t \left(K \|u\|_{L^\infty}  A T e^{BT} x +  B \right),
\end{align*}
where $K$ is the Lipschitz constant of $\rho'$. Since $x\leq R_\varepsilon$, let $C_G(T,\varepsilon) = K \|u\|_{L^\infty}  A T e^{BT} R_\varepsilon +  B$, and  if $|x| \leq y$, then

\[
 \left| e^x - 1 \right| \leq \left| e^y-1 \right|  + \left| e^{-y}-1  \right|.
\]
Thus, we get
\begin{equation} \label{eq:estimI3_3}
 J_3^2  \leq  e^{BT}  \|\fin_\varepsilon \|_{L^\infty} \left(\big| e^{\delta t C_G(T,\varepsilon)} -1 \big| + \big| e^{- \delta t C_G(T,\varepsilon)} -1 \big|  \right) \int_{x_0}^{R_\varepsilon} x^r dx.
\end{equation}
Since $\mu$ is nonnegative, $J_3^3 \leq$
\begin{equation*}
e^{BT}  \|\fin_\varepsilon \|_{L^\infty} \int_{X(t+\delta t;x_0,0)}^{R_\varepsilon} x^r  \left|e^{-\left(\int_{0}^{t+\delta t} \mu(X(\sigma;x,t+\delta t))
d\sigma - \int_{0}^t \mu(X(\sigma;x,t))d\sigma \right)}  -  1\right|  dx. %%%%%%%%%%%%%%
\end{equation*}
Exactly as above,
\begin{equation*}
\Big|\int_{0}^{t+\delta t} \mu(X(\sigma;x,t+\delta t))d\sigma - \int_{0}^t \mu(X(\sigma;x,t))d\sigma  \Big| \leq \delta t M A T e^{BT} x + \delta t  \| \mu \|_{L^\infty},
\end{equation*}
with $M =$  Lipschitz constant of $\mu$. Denoting by $C_M(T,\varepsilon) = M A T e^{BT} R_\varepsilon + \| \mu \|_{L^\infty}$, we get
\begin{align} \label{eq:estimI3_4}
J_3^3 \leq  e^{BT}  \|\fin_\varepsilon \|_{L^\infty} \left(\big| e^{\delta t C_M(T,\varepsilon)} -1 \big| + \big| e^{- \delta t C_M(T,\varepsilon)} -1 \big|  \right) \int_{x_0}^{R_\varepsilon} x^r dx.
\end{align}
From  \eqref{eq:estimI3_1}, \eqref{eq:estimI3_2}, \eqref{eq:estimI3_3} and \eqref{eq:estimI3_4}  we can conclude that for any   $\varepsilon>0$,
\begin{equation} \label{eq:EstimI3}
I_3(\delta t) \leq C_3^1(T,\varepsilon) + C_3^2(T,\delta t,\varepsilon),
\end{equation}
with $\lim_{\varepsilon \rightarrow 0} C_3^1(T,\varepsilon) = 0$ and $\lim_{\delta t \rightarrow 0} C_3^2(T,\delta t,\varepsilon) = 0$.

Next, concerning $I_1$, $f$ can be written from the boundary condition. Let $u^\varepsilon \, \in \, \mathcal C^\infty_0$ such that
$u^\varepsilon \longrightarrow u$ uniformly on $[0,T]$.
Then we write $I_1$ as follows:
\begin{align*}
I_1 & \leq \int_{x_0}^{X(t+\delta t;x_0,0)} x^r | N (u(s_0(x,t+\delta t))  - N (u^\varepsilon(s_0(x,t+\delta t))| I(x,t+\delta t)  dx \\
& \begin{aligned} \quad + \int_{x_0}^{X(t;x_0,0)} x^r   & \left | N (u^\varepsilon(s_0(x,t+\delta t))  I(x,t+\delta t) e^{-\int_{s_0(x,t+\delta t)}^t \mu(X(\sigma;x,t+\delta t))d\sigma} \right. \\
& \left. - N (u^\varepsilon(s_0(x,t)) I(x,t)e^{-\int_{s_0(x,t)}^t \mu(X(\sigma;x,t))d\sigma}\right| dx
\end{aligned}\\
& \quad + \int_{x_0}^{X(t;x_0,0)} x^r | N (u(s_0(x,t))  - N (u^\varepsilon(s_0(x,t))| I(x,t)  dx.
\end{align*}
From (H3) we obtain, similarly to $I_3$, that there exist two constants $C_1^1(T,\varepsilon) $ and $C_1^2(T,\delta t,\varepsilon)$ such that
\begin{equation} \label{eq:EstimI1}
I_1(\delta t) \leq C_1^1(T,\varepsilon) + C_1^2(T,\delta t,\varepsilon),
\end{equation}
with $\lim_{\varepsilon \rightarrow 0} C_1^1(T,\varepsilon) = 0$ and $\lim_{\delta t \rightarrow 0} C_1^2(T,\delta t,\varepsilon) = 0$.

Finally, for $I_2$, we use the two formulas of $f$,
\begin{multline*}
I_2=
\int_{X(t;x_0,0)}^{X(t+\delta t;x_0,0)} x^r \left| N(u(s_0(x,t+\delta t)))I(x,t+\delta t) e^{-\int_{s_0(x,t+\delta t)}^{t+\delta t} \mu(X(\sigma;x,t+\delta t))d\sigma}  \right. \hfill \\
\hfill \left. - \fin(X(0;x,t))J(0;x,t)e^{-\int_{s_0(x,t)}^{t} \mu(X(\sigma;x,t))d\sigma} \right| dx
\end{multline*}
Using the Lipschitz constant of $N$ denoted by $K_N$, from the definition of $I$ and with the help of lemma \ref{lem:carac}, we get
\begin{multline*}
I_2 \leq   x_0^r e^{(rA+B)T} K_N \left| {X(t+\delta t;x_0,0)} - X(t;x_0,0) \right| \hfill \\
\hfill + x_0^r e^{rAT} \int_{X(t;x_0,0)}^{X(t+\delta t;x_0,0)} \left| \fin(X(0;x,t))J(0;x,t) \right| dx.
\end{multline*}
Using the regularization $\fin_\varepsilon$ of $\fin$, there exist two constants $C_2^1(T,\varepsilon) $ and $C_2^2(T,\delta t,\varepsilon)$ such that for any $\varepsilon>0$,
\begin{equation} \label{eq:EstimI2}
I_2(\delta t) \leq C_2^1(T,\varepsilon) + C_2^2(T,\delta t,\varepsilon),
\end{equation}
with $\lim_{\varepsilon \rightarrow 0} C_2^1(T,\varepsilon) = 0$ and $\lim_{\delta t \rightarrow 0} C_2^2(T,\delta t,\varepsilon) = 0$.

In conclusion, combining \eqref{eq:EstimI3}, \eqref{eq:EstimI1} and \eqref{eq:EstimI2}, we get for any $\varepsilon>0$ and $ \delta t >0$,
\begin{equation*}
 \int_{x_0}^{+\infty} x^r |f(x,t+\delta t) - f(x,t)| dx  \leq C^1(T,\varepsilon) + C^2(T,\delta t,\varepsilon),
\end{equation*}
where $C^1(T,\varepsilon) $ and $C^2(T,\delta t,\varepsilon)$ are two constants such that $\lim_{\varepsilon \rightarrow 0} C^1(T,\varepsilon) = 0$ and $\lim_{\delta t \rightarrow 0} C^2(T,\delta t,\varepsilon) = 0$. Noticing that the proof remains the same when $\delta t$ is negative, taking the $\limsup$ in $\delta t$ we get
\begin{equation*}
 0 \leq \limsup_{\delta t \rightarrow 0} \int_{x_0}^{+\infty} x^r |f(x,t+\delta t) - f(x,t)| dx  \leq C^1(T,\varepsilon), \text{ for any } \varepsilon > 0.
\end{equation*}
The proof is completed by taking the limit as $\varepsilon$ goes to zero, which yields to the require regularity, $f\in\mathcal C^0([0,T],L^1([x_0,+\infty),x^rdr)$ for all $r\in[0,1]$.
\end{proof}
\noindent We finish this section with a useful estimate for the uniqueness investigation.
\begin{proposition} \label{prop:contract}
Let $T>0$ and $u_1$, $ u_2 \, \in \, \mathcal C^0_b(0,T)$. Let $f_1$ and $f_2$ be two mild solutions to \eqref{eq:transport}-\eqref{eq:initial_f_auto}, associated, respectively to $u_1$ and $u_2$, with initial data $\fin_1$,  $\fin_2$ given by formula \eqref{eq:mild}.  Then, for any $t\in(0,T)$
%
% \begin{equation*}
% \begin{array}{ll}
% &\displaystyle \int_{x_0}^{+\infty} x \left| f_1(x,t) -  f_2(x,t) \right| dx
%  \leq \int_{x_0}^{+\infty} x \left| \fin_1(x)- \fin_2(x) \right| dx
% \vspace{0.1cm}\\
% & \,  - \displaystyle \int_0^t \int_{x_0}^{+\infty} \mu(x) x \left| \fin_1(x,s) - \fin_2(x,s) \right| dx ds \vspace{0.1cm}\\
% & \, +  A_1 \displaystyle \int_0^t \int_{x_0}^{+\infty} x \left| f_1(x,s) - f_2(x,s) \right| dx ds
% \vspace{0.1cm}\\
% & +\displaystyle  \int_0^t \left( K_{1,2} + C \| f_2(\cdot,s)\|_{L^1(xdx)} \right) \left| u_1(s) - u_2(s)\right| ds,
% \end{array}
% \end{equation*}

\begin{multline*}
\int_{x_0}^{+\infty} x \left| f_1(x,t) -  f_2(x,t) \right| dx  \leq \int_{x_0}^{+\infty} x \left| \fin_1(x)- \fin_2(x) \right| dx  \hfill \\
-  \int_0^t \int_{x_0}^{+\infty} \mu(x) x \left| \fin_1(x,s) - \fin_2(x,s) \right| dx ds   +  A_1  \int_0^t \int_{x_0}^{+\infty} x \left| f_1(x,s) - f_2(x,s) \right| dx ds \\
\hfill + \int_0^t \left( K_{1,2} + C \| f_2(\cdot,s)\|_{L^1(xdx)} \right) \left| u_1(s) - u_2(s)\right| ds,
\end{multline*}
where $A_1$ is given by \eqref{eq:a_slin} for $u_1$ and $K_{1,2}$ is the Lipschitz constant of $N$ on $[0,R]$ with $R= \max (\|u_1\|_{L^\infty(0,T)},$ $\|u_2\|_{L^\infty(0,T)})$. Finally $C>0$ denotes a constant such that $\rho(x)<Cx$.
\end{proposition}
\begin{proof}
This estimation is obtained from a classical argument of approximation. Let $h = f_1 - f_2$ and
\begin{multline*}
\int_{x_0}^{+\infty} h(x,t)\varphi(x,t) dx = \int_{x_0}^{+\infty} h^{in}(x)\varphi(x,0) dx  + \int_0^t \left( N(u_1(s)) - N(u_2(s)) \right) \varphi(x_0,s) ds  \hfill  \\
+ \int_0^t\int_{x_0}^{+\infty}  h(x,s) \left[ \frac{\partial}{\partial t}\varphi(x,s) + a_1(s,x) \frac{\partial}{\partial x}\varphi(x,s) - \mu(x) \varphi(x,s)\right] dx ds\\
\hfill + \int_0^t\int_{x_0}^{+\infty} \left( a_1(s,x) - a_2(s,x) \right) f_2(x,s) \frac{\partial}{\partial x}\varphi(x,s) dx ds.
\end{multline*}
Let $h_\varepsilon$ be a regularization of $h$ and $S_\delta$ a regularization of the $Sign$ function. Take $\varphi(x,s) = S_\delta(h_\varepsilon(s,x))g(x)$ with $g\in \mathcal C^\infty_c([x_0,+\infty))$. Then, letting $\delta \rightarrow 0$ and then $\varepsilon \rightarrow 0$, we get
\begin{equation*}
\begin{array}{rcl}
\ds \int_{x_0}^{+\infty} |h(x,t)|g(x) dx & = & \ds \int_{x_0}^{+\infty} |h^{in}(x)|g(x) dx \\
 & &\ds + \int_0^t \left| N(u_1(s)) - N(u_2(s)) \right) Sign(h_0(x_0)) g(x_0) ds \\
 & &\ds  + \int_0^t\int_{x_0}^{+\infty} |h(x,s)| \left[a_1(s,x) \frac{\partial}{\partial x}g(x)   -\mu(x)  g(x) \right] dx ds\\
 & & \ds  + \int_0^t\int_{x_0}^{+\infty} \left( a_1(s,x) - a_2(s,x) \right) f_2(x,s) Sign(h(s,x))\frac{\partial}{\partial x}g(x) dx ds.
 \end{array}
\end{equation*}
Finally, we approximate the identity function with a regularized function $\eta_R\in \mathcal C^\infty_c([x_0,+\infty))$ such that $\eta_R(x) = x$ over $(0,R)$, and  then taking  the limit $R \rightarrow +\infty$ ends the proof.

\end{proof}
We get straightforward from proposition \ref{prop:weak} that $f$ defined by \eqref{eq:mild} is a weak solution and the only one from proposition  \ref{prop:contract}. Indeed, getting $u_1 =u_2$ and $f_1^0 = f_2^0$ in proposition \ref{prop:contract} leads to the uniqueness. Finally, proposition \ref{prop:continuity} provides the continuity in time of the moments with order less or equal to one. This concludes the proof of proposition \ref{lem:auto}

\subsection{Proof of the well-posedness}\label{subsec:proof}

In this section we prove theorem \ref{thm:wellposed}. We first study the operator $S$ in \eqref{eq:operator}.

\begin{lemma}\label{lem:contraction}
Consider hypothesis (H2) to (H4). Let $\uin$, $\pin$ and $\bin$ be nonnegative  initial data, and let $\fin$ satisfy (H1).
Let $M >0$ be large enough such that $\uin, \pin, \bin < M/2$ and define
\begin{equation*}
X_M = \left\{ (u,p,b)\in\mathcal{C}^0([0,T])^3 : 0 \leq u,p,b \leq M\right\}
\end{equation*}
where $\mathcal{C}^0([0,T])^3$ is equipped with the uniform norm. Then, there exists $T>0$ (small enough) such that $S : X_M  \mapsto  X_M$  is a contraction.
\end{lemma}
\begin{proof}
Let $M$ be sufficiently large such that $\max(\uin,\pin,\bin) < M/2$, and let $T>0$ be small enough such that
\begin{align*}
&(\gamma_u  +\tau M  + \sigma + x_0 C_1(M) +  C_2(M,T)) M T \leq M/2, \\ %\label{eq:ineq1}\\
&(\gamma_p  +\tau M) M T \leq M/2, \\%\label{eq:ineq2} \\
&(\sigma+\delta) M T \leq M/2,\\% \label{eq:ineq3} \\
&(\lambda_u + \sigma M) T  \leq M/2,\\%  \label{eq:ineq12}\\
&(\lambda_p +\sigma M ) T \leq M/2, \\%\label{eq:ineq22} \\
&\tau M^2 T \leq M/2, %\label{eq:ineq32}
\end{align*}
where $C_1(M)$ is the Lipschitz constant of $N$ on $(0,M)$ and
\begin{equation}
C_2(M,T) = C e^{M C T}\left( \|\fin\|_{L^1(xdx)} + C_1(M)M T\right),
\end{equation}
where $C$ is the constant such that  $\rho(x)\leq Cx$, see \eqref{eq:intx}. This assumption ensures that for any $(u,p,b) \in X_M$, then $S(u,p,b) \in X_M$, \emph{i.e}, the solution is bounded by $M$ and is nonnegative.
It remains to prove that $S$ is a contraction. Let $(u_1,p_1,b_1)$ and $(u_2,p_2,b_2)$ belong to $X_M$.  Then
\begin{multline}
\|S_{u_1}- S_{u_2}\|_\infty \leq  \gamma_u T \|u_1-u_2\|_\infty + \tau T \|u_1p_1 - u_2p_2\|_\infty + \sigma T \|b_1-b_2\|_\infty  \hfill\\
\vphantom{ \Bigg| }  + x_0 T C_1(M)\|u_1-u_2\|_\infty \\
\hfill + T \sup_{t\in[0,T]} \left|u1 \int_{x_0}^{+\infty} \rho(x) f_1(x,s) dx -u2 \int_{x_0}^{+\infty} \rho(x) f_2(x,s) dx \right|.
\end{multline}
Then,
\begin{equation}
\| u_1 p_1 - u_2p_2 \|_\infty  \leq M \|u_1-u_2\|_\infty + M \|p_1-p_2\|_\infty,
\end{equation}
\begin{multline}
\sup_{t\in[0,T]} \left|u1 \int_{x_0}^{+\infty} \rho(x) f_1(x,s) dx -u2 \int_{x_0}^{+\infty} \rho(x) f_2(x,s) dx \right| \hfill \\
\hfill \leq C_2(M,T) \|u_1 - u_2\|_\infty + C M \sup_{t\in[0,T]} \left| \int_{x_0}^{+\infty} x |f_1(x,t) - f_2(x,t)| dx \right|,
\end{multline}
and from Proposition \ref{prop:contract},
\begin{equation}
\sup_{t\in[0,T]} \left| \int_{x_0}^{+\infty} x |f_1(x,t) - f_2(x,t)| dx \right| \leq T \left(C_1(M)+ C C_2(M,T)\right)\|u_1 - u_2\|_\infty.
\end{equation}
We get similar bounds for $|S_{p_1}- S_{p_2}|_\infty$ and $|S_{b_1}- S_{b_2}|_\infty$. We infer that there exists a constant $C(M,T)$ depending only on $M$ and $T$ such that
\begin{equation}
\|(S_{u_1},S_{p_1},S_{b_1}) - (S_{u_2},S_{p_2},S_{b_2})\|_\infty \leq  C(M,T) T \|( u_1, p_1, b_1) - ( u_2, p_2, b_2)\|_\infty,
\end{equation}
with $C(M,T) T  \rightarrow 0$, when $T$ goes to 0. Hence, if $T$ is small enough such that $C(M,T) T <1$ , then $S$ is a contraction.
\end{proof}
From Lemma \ref{lem:contraction}, we have a local nonnegative solution on $[0,T]$, which is unique with  the solution $(u,p,b)$ bounded by the constant $M$. The solution satisfies $f\in C^0(0,T;L^1(xdx))$ and $u,p,b\in C^0(0,T)$. Futhermore from (H3), $N$ is continuous and from (H2), $\rho(x) \leq Cx$ where C is a positive constant. Thus $\rho f \in C^0(0,T;L^1(dx))$. We conclude that $u,p$ and $b$ defined in definition \ref{def:exist} have continuous derivatives.

Now we remark that the solutions satisfies on $[0,T]$

\begin{align*}
\frac{ \mathrm d}{\mathrm d t} \left( u + p + 2b \right) &  = \lambda_u + \lambda_p - \gamma_u u - \gamma_p p - \delta 2 b  - n N(u) \\
& - \frac 1 \varepsilon u \int_{x_0}^{+\infty} \rho(x) f(x,t) d x
 \leq \lambda - m (u + p + 2b),
\end{align*}
with $m = \min(\gamma_u,\gamma_p,\delta)$ and $\lambda = \lambda_u + \lambda_p$. Using Gronwall's lemma,  the solutions remain bounded, at any time by
\begin{equation}
u + p + 2b \leq \uin + \pin + 2\bin + \frac \lambda m.
\end{equation}
From this global bound on $u,p$ and $b$, we can construct the solution on any interval of time by repetition of the local argument. The proof of the theorem is complete.

\section{Summary}

The connection of prions and AD is not fully understood, but recent research suggests that soluble A$\beta$ oligomers are possible inducers of AD neuropathology. The key element of this hypothesis is that $\beta-$amyloid plaques increase their size over disease progression by the clustering of A$\beta$ oligomers, which are bound to \PRPC proteins. A$\beta$ oligomers exist both as bounded and unbounded to \PRPC proteins, and the agglomerartion rate in the formation of  $\beta-$amyloid plaques  depends on the concentrations of the bound and unbound A$\beta$ oligomers, the concentration of soluble  \PRPC, and the size of the $\beta-$amyloid plaques. We have introduced a mathematical model of the evolution of AD based on these hypotheses, and presented a  mathematical analysis of its fundamental properties. Specifically, we have analyzed in detail the existence and uniqueness properties of solutions, as well as the qualitative properties of solution behavior. In specific cases we have quantified the
stabilization of the solutions to steady state, a well-known feature of AD progression. In future work we will explore applications of this model to specific AD laboratory and clinical data.

\appendix
%\section{}

\section{Characteristic polynomials of the linearized ODE system}
Here we give the coefficient $a_i$, $i=1,\ldots,4$ for the characteristic polynomial of the linearized system in proposition \ref{thm:equilibrium}:
% \begin{align*}
%   a_1 & =  \left(\mu+\gamma_u+\tau \frac{\lambda_p}{\tau^*u_\infty+\gamma_p}+\alpha n^2 u_\infty^{n-1}+\rho \frac{\alpha}{\mu}u_\infty^n+\gamma_p+\tau u_\infty+\sigma+\delta\right),\\
%   a_2 & =  \left(\mu+\gamma_u+\alpha n^2u_\infty^{n-1}+\rho \frac{\alpha}{\mu}u_\infty^n\right)(\gamma_p+\tau u_\infty+\sigma+\delta)+\gamma_p\sigma+(\gamma_p+\tau u_\infty)\delta \nonumber \\
%   &\hphantom{ =  } \ +  \mu\left(\gamma_u+\tau \frac{\lambda_p}{\tau^*u_\infty+\gamma_p}+\alpha n^2u_\infty^{n-1}+\rho \frac{\alpha}{\mu}u_\infty^n\right) \\
%   &\hphantom{ =  } \ +\rho\alpha nu_\infty^n+\tau(\gamma_p+\delta)\frac{\lambda_p}{\tau^*u_\infty+\gamma_p},  \\
%   a_3 & =  \left(\mu+\gamma_u+\alpha n^2 u_\infty^{n-1}+\rho \frac{\alpha}{\mu} u_\infty^n\right)(\gamma_p\sigma+(\gamma_p+\tau u_\infty)\delta)+(\gamma_p\delta \\
%   &\hphantom{ =  } \ +(\gamma_p+\delta)\mu)\tau \frac{\lambda_p}{\tau^* u_\infty+\gamma_p} \nonumber \\
%     & \hphantom{ =  } \ +  \left\{\mu\left(\gamma_u+\alpha n^2u_\infty^{n-1}+\rho \frac{\alpha}{\mu}u_\infty^n\right)+\rho\alpha nu_\infty^n\right\}(\gamma_p+\tau u_\infty+\sigma+\delta), \\
%   a_4 & =   \mu\gamma_p\delta\tau \frac{\lambda_p}{\tau^*u_\infty+\gamma_p} +\left\{\mu\left(\gamma_u+\alpha n^2u_\infty^{n-1}+\rho \frac{\alpha}{\mu}u_\infty^n\right) +\rho\alpha nu_\infty^n\right\} \\
%  &\hphantom{ =  } \  \times
%   (\gamma_p\sigma+(\gamma_p+\tau u_\infty)\delta).
% \end{align*}

\begin{align*}
  a_1 & =  \left(\mu+\gamma_u+\tau \frac{\lambda_p}{\tau^*u_\infty+\gamma_p}+\alpha n^2 u_\infty^{n-1}+\rho \frac{\alpha}{\mu}u_\infty^n+\gamma_p+\tau u_\infty+\sigma+\delta\right),\\
  a_2 & =  \left(\mu+\gamma_u+\alpha n^2u_\infty^{n-1}+\rho \frac{\alpha}{\mu}u_\infty^n\right)(\gamma_p+\tau u_\infty+\sigma+\delta)+\gamma_p\sigma+(\gamma_p+\tau u_\infty)\delta \nonumber \\
  &\hphantom{ =  } \ +  \mu\left(\gamma_u+\tau \frac{\lambda_p}{\tau^*u_\infty+\gamma_p}+\alpha n^2u_\infty^{n-1}+\rho \frac{\alpha}{\mu}u_\infty^n\right) +\rho\alpha nu_\infty^n+\tau(\gamma_p+\delta)\frac{\lambda_p}{\tau^*u_\infty+\gamma_p},  \\
  a_3 & =  \left(\mu+\gamma_u+\alpha n^2 u_\infty^{n-1}+\rho \frac{\alpha}{\mu} u_\infty^n\right)(\gamma_p\sigma+(\gamma_p+\tau u_\infty)\delta)+(\gamma_p\delta  +(\gamma_p+\delta)\mu)\tau \frac{\lambda_p}{\tau^* u_\infty+\gamma_p} \nonumber \\
    & \hphantom{ =  } \ +  \left\{\mu\left(\gamma_u+\alpha n^2u_\infty^{n-1}+\rho \frac{\alpha}{\mu}u_\infty^n\right)+\rho\alpha nu_\infty^n\right\}(\gamma_p+\tau u_\infty+\sigma+\delta), \\
  a_4 & =   \mu\gamma_p\delta\tau \frac{\lambda_p}{\tau^*u_\infty+\gamma_p} +\left\{\mu\left(\gamma_u+\alpha n^2u_\infty^{n-1}+\rho \frac{\alpha}{\mu}u_\infty^n\right) +\rho\alpha nu_\infty^n\right\}  (\gamma_p\sigma+(\gamma_p+\tau u_\infty)\delta).
\end{align*}

\section{Lyapunov functionnal}

In order to obtain a Lyapunov function for system (\ref{eq:edo1}-\ref{eq:edo4}), we approached the problem in a backward manner as described in \cite{Khalil2002}*{Chap. 4 - p. 120}. We investigated an expression for the derivative $\dot\Phi$ and went back to chose the parameters of $\Phi$ so as to make $\dot\Phi$ negative definite. This is useful idea to searching for a lyapunov function. Using this method we can derive the global stability in proposition \ref{prop:global_stab}. The Lyapunov function for the system (\ref{eq:edo1}-\ref{eq:edo4}) is:
\begin{equation*}
\begin{array}{rcl}
\ds \Phi & = &\ds \frac{1}{2}\left(\frac{2\gamma_p}{\delta}\right)s_1\theta_1^2+\frac{1}{2}\left(1+2\frac{\delta+\gamma_u+\rho(A_\infty+\theta_1)}
{\sigma}\right)\theta_2^2+\frac{1}{2}\left(\frac{2\gamma_p}{\delta}\right)\theta_3^2 \\
& & \ds  +\frac{1}{2}\left(\frac{\sigma}{\gamma_p}\right)\theta_4^2 + \left(\frac{\rho p_\infty}{\gamma_u+\rho A_\infty+\mu}\right)\theta_1\theta_2+\theta_1\theta_3+\theta_2\theta_3\\
& &\ds +\left(\frac{\rho p_\infty}{\gamma_u+\rho A_\infty+\mu}+1+\frac{\rho}{\tau}\right)\theta_1\theta_4+2\theta_2\theta_4+\left(\frac{2\gamma_p}{\delta}\right)\theta_3\theta_4,
\end{array}
\end{equation*}
where $\theta_1=A-A_\infty$, $\theta_2=u-u_\infty$, $\theta_3=p-p_\infty$, $\theta_4=b-b_\infty$, with $s_1=\max(T_1,T_2)$ such that
\begin{equation*}
T_1=\frac{\rho^2\delta u_\infty^2(1+2\frac{1+\delta}{\sigma})}{8\mu\gamma_p}+\frac{(\gamma_p+\mu)^2(\frac{\delta}{2\gamma_p})^2}{4\gamma_p\mu} +\frac{[(\delta+\mu)(\frac{\rho p_\infty}{\gamma_u+\rho A_\infty+\mu}+1) +(\sigma+\delta+\mu)\frac{\rho}{\tau}+2\rho u_\infty]^2}{8\mu\sigma},
\end{equation*}

\begin{equation*}
\begin{array}{rcl}
\ds T_2 & = & \ds \frac{\left(\frac{\delta}{2\gamma_p}\right)^2\left(\frac{\rho p_\infty}{\gamma_u+\rho A_\infty+\mu}\right)^2\left(\frac{2\sigma+\delta}{2\gamma_p}\right)}
{\left(1+2\frac{\delta+\gamma_u}{\sigma}-\frac{\delta}{2\gamma_p}\right)\left(\frac{\delta}{2\gamma_p}\frac{\sigma}{\gamma_p}-1\right)}
+\frac{\left(\frac{\delta}{2\gamma_p}\right)^2\left(\frac{\rho p_\infty}{\gamma_u+\rho A_\infty+\mu}\right)\left[2+4\frac{\rho}{\tau}\frac{\delta+\gamma_u}{\sigma}\right]}
{\left(1+2\frac{\delta+\gamma_u}{\sigma}-\frac{\delta}{2\gamma_p}\right)\left(\frac{\delta}{2\gamma_p}\frac{\sigma}{\gamma_p}-1\right)} \\
& & \ds +\frac{\left(\frac{\delta}{2\gamma_p}\right)^3\left[\frac{\rho}{\tau}\left(2+\frac{\rho}{\tau}\right)+
\frac{\sigma}{\gamma_p}+2\frac{\delta+\gamma_u}{\gamma_p}\right]}
{\left(1+2\frac{\delta+\gamma_u}{\sigma}-\frac{\delta}{2\gamma_p}\right)\left(\frac{\delta}{2\gamma_p}\frac{\sigma}{\gamma_p}-1\right)}
+\frac{\left(\frac{\delta}{2\gamma_p}\right)^2\left(1+\frac{\rho}{\tau}\right)\left[1+2\frac{\delta+\gamma_u}{\sigma}\right]\frac{\rho}{\tau}}
{\left(1+2\frac{\delta+\gamma_u}{\sigma}-\frac{\delta}{2\gamma_p}\right)\left(\frac{\delta}{2\gamma_p}\frac{\sigma}{\gamma_p}-1\right)} \\
& & \ds +\left(\frac{\delta}{2\gamma_p}\right)\left(\frac{\rho p_\infty}{\gamma_u+\rho A_\infty+\mu}\right)^2\left(\frac{1}
{1+2\frac{\delta+\gamma_u}{\sigma}}\right)
+\frac{\left(1+2
\frac{\delta+\gamma_u}
{\sigma}\right)\left(\frac{\delta}{2\gamma_p}\right)^2
}{\left(1+2\frac{\delta+\gamma_u}{\sigma}-\frac{\delta}{2\gamma_p}\right)
}.
\end{array}
\end{equation*}
This Lyapunov function $\Phi$ is positive when $\left(1+2\frac{\delta+\gamma_u}{\sigma}\right)>\frac{\delta}{2\gamma_p}>\frac{\gamma_p}{\sigma}$.
Its derivative along the solutions to the system (\ref{eq:edo1}-\ref{eq:edo4}) is
\begin{equation*}
\begin{array}{rcl}
 \ds\dot \Phi & = & \ds  -\displaystyle\left(\mu s_1+\rho u\frac{\rho\frac{\delta}{2\gamma_p}p_\infty}{\gamma_u+\rho A_\infty+\mu}\right)\theta_1^2-\rho u_\infty\left(1+2\frac{\gamma_u+\rho(A_\infty+\theta_1)+\delta}{\sigma}\right)\left(\frac{\delta}{2\gamma_p}\right)\theta_1\theta_2 \\
   &    & \ds -  \displaystyle\left(\frac{2(\gamma_u+\rho(A_\infty+\theta_1)+\tau p)(\gamma_u+\rho(A_\infty+\theta_1)+\delta)}{\sigma}+\gamma_u+\rho(A_\infty+\theta_1)\right)\left(\frac{\delta}{2\gamma_p}\right)\theta_2^2 \\
       & & \ds -  \displaystyle\left((\delta+\mu)\left(\frac{\rho p_\infty}{\gamma_u+\rho A_\infty+\mu}+1\right)+(\sigma+\delta+\mu)\frac{\rho}{\tau}+2\rho u_\infty\right)\left(\frac{\delta}{2\gamma_p}\right)\theta_1\theta_4\\
       & & \ds -  \displaystyle\left(\frac{\delta\tau u}{2\gamma_p}+\gamma_p\right)\theta_3^2-\delta\left(\frac{\sigma}{\gamma_p}\frac{\delta}{2\gamma_p}\right)\theta_4^2 - (\gamma_p+\mu)\left(\frac{\delta}{2\gamma_p}\right)\theta_1\theta_3.
       \end{array}
\end{equation*}
$\dot \Phi$ is  non-positive. Furthermore, $\dot \Phi=0$ if and only if $\theta_1=\theta_2=\theta_3=\theta_4=0$. LaSalle's invariant principle  then implies that the unique equilibrium is globally asymptotically stable in the stable subset defined in (\ref{eq:stableset}) \cite{Lasalle1976}.

%%% Biblio %%%

%\bibliographystyle{amsplain}
%\bibliography{bibli_alzh}


%%% Signature %%%

\vspace{1.5cm} \scriptsize \parskip=0pt

\centering \rule{0.33\textwidth}{.5pt}

\vspace{1.5cm}

\flushleft
\signmh

\flushright
\signeh

\flushleft
\signlpm

\flushright
\signgfw

\end{document}